%% file: main.tex
\documentclass[a4paper, 12pt, reqno]{amsart}
\usepackage[utf8]{inputenc}

\input{macros}

\numberwithin{foo}{section}
\numberwithin{equation}{section}

\title{Curvature Estimates for the Continuity Method}

\begin{document}
\date{\today}
\author{Hosea Wondo}

\maketitle

\begin{abstract}
    We obtain curvature estimates for long time solutions of the continuity method on compact K\"ahler manifolds with semi-ample canonical line bundles. In this setting, initiated in \cite{LT16} and \cite{R08}, we adapt arguments from \cite{FZ20} for the K\"ahler-Ricci flow to this setup. As an application, we derive curvature bounds for general metrics on product manifolds. 
\end{abstract}

\section{Introduction}
\blfootnote{2020 Mathematics Subject Classification. Primary 53C55; Secondary 32Q15.}
Let $\omega(t)$ be a smooth family of K\"ahler metrics on a K\"ahler Manifold $(X,\omega_0)$, parameterized by $t \geqslant 0$, such that
\begin{equation}\label{cont}
\begin{cases}
        (1+t)\omega(t) = \omega_0 - t \Ric (\omega),  \\
        \omega(0) = \omega_0.
\end{cases}
\end{equation}

This equation, known as the continuity method, first appeared in \cite{R08} by Rubinstein, and was later reintroduced by La Nave and Tian in \cite{LT16} as an alternative to the K\"ahler-Ricci flow for carrying out the analytic minimal model program \cites{ST12,ST17,ST06,T19}. The main advantage the continuity method has over the flow is the immediate lower bound for the Ricci curvature along the deformation. This allows the application of geometric comparison techniques. For the K\"ahler-Ricci flow, a Ricci bound is conjectured in general, but this problem is still open (for some recent progress, see \cites{FZ20}). 

As with the K\"ahler-Ricci flow, the continuity method should deform any metric on a numerically effective K\"ahler manifold to its canonical metric, identified to be one of K\"ahler-Einstien type. At the level of Cohomology, \eqref{cont} deforms the initial class $[\omega_0]$ to the manifold's first Chern class $-c_1(X)$. Motivated by the abundance conjecture, we assume that canonical line bundle, $K_X$, is semi-ample. Under this assumption,  $K_X$ generates a fiber space map 
\begin{equation}\label{fmap}
    f : X \rightarrow X_{can} \subset \CC \PP^N. 
\end{equation}
This map may contain singular fibers, thus we denote 
$$\Sigma_s := \{s \in \Sigma \mid X_s:= f^{-1}(s) \text{ is a singular fiber}\}$$ to be the singular set resulting from $f$. It turns out that the convergence of the family of metrics satisfying the continuity method depends on the kodaria dimension of $X$, kod$(X)$. Geometrically, this turns out to be the dimension of the possibly singular manifold $\Sigma$ away from the singular set $\Sigma_s$. The Kodaria dimension cannot exceed the complex dimension of $X$, $n$. We can therefore distinguish the following three cases depending on kod$(X)$: 
\begin{enumerate}
    \item We say that $X$ is of \textit{General Type} if kod$(X) = n$. It was shown in \cites{LT16} (after rescaling appropriately) that the metric weakly converges to a K\"ahler-Einstien metric as $t \rightarrow \infty$. 
    \item If kod$(X) = 0$, then $X$ is a \textit{Calabi Yau manifold}. For the K\"ahler-Ricci flow, classical result by Cao in \cites{C86} can be carried over to the continuity method to show smooth deformation of $(1+t) \omega(t)$ to the unique Ricci flat metric in the class $[\omega_0]$ as $t \rightarrow \infty$.
    \item If  $0 <$ kod$(X) < n$, then we are in the \textit{Volume Collapsed Case}. For the K\"ahler-Ricci flow, the volume collapsed case has been extensively studied \cites{STZ19, TZ21,TZ15, yZ19, Z20,CL21,GTZ20,FZ20,GTZ13,JS21,yZ19,TWX18}. When $X$ admits a Fano fibration, Gromov-Hausdorff convergence for \eqref{cont} is obtained in \cite{ZZ20}. When $X$ is a minimal elliptic surface, similar convergence was obtained in \cite{ZZ19}. In the more general setting where $X$ is a Hermitian manifold admitting an elliptic bundle over a Rienmann surface, Gromov-Hausdorff convergence and curvature estimates are established in \cite{SW20}. 
\end{enumerate}

 It was shown in \cite{Z20} that the singularity type of the K\"ahler-Ricci flow does not depend on the choice of the initial metric, from which Y. Zhang gives an elegant proof of some classification results by Tossati-Zhang \cite{TZ15}. The authors Fong and Y. Zhang in \cite{FZ20} then obtained an improved result which localizes the estimate and allows for degenerate metric and curvature bounds. The results of this paper are inspired by results from these two papers. 

\begin{thm}\label{thm1}
    Suppose that we are in the setup given by \eqref{fmap}. Assume there exists  $\tau: [0,\infty) \rightarrow [1,\infty)$ and an initial metric $\widetilde{\omega}_0$ whose solution $\widetilde{\omega}(t)$ to \eqref{cont} satisfies 
        \begin{equation}\label{a1}
        \sup_{X}|Rm(\widetilde{\omega}(t))|_{\widetilde{\omega}(t)} \leqslant \tau(t).
    \end{equation}
Then the solution $\omega(t)$ to \eqref{cont} starting from any other initial metric $\omega(0) = \omega_0$ satisfies 
    \begin{equation}\label{meq}
         C^{-1}e^{- C\tau(t)} \widetilde{\omega}(t) \leqslant \omega(t) \leqslant C e^{C \tau(t)} \widetilde{\omega}(t),
    \end{equation}
    for all $t \geqslant 0$ and for some $C>0$ uniform of $t$. 
\end{thm}

From the metric equivalence \eqref{meq}, we can also obtain Ricci curvature bounds for \eqref{cont}.

\begin{cor}\label{cor1}
In the same setting as in Theorem \ref{thm1}, there exists $C>0$ uniform of $t$ such that 
\begin{equation}\label{Riccor}
 -C \omega(t) \leqslant \Ric(\omega(t)) \leqslant C e^{C \tau(t)} \omega(t). 
\end{equation}
\end{cor}
    
Therefore, if one can find a solution with uniformly bounded curvature, then any other solution starting from a different initial metric will be equivalent and will have bounded Ricci curvautre. For the curvature tensor, we have the following theorem. 

\begin{thm}\label{thm2}
Suppose that we are in the setup given by \eqref{fmap}. Assume there exists a family of metrics $\widetilde{\omega}$ solving \eqref{cont} and  functions $\tau,\sigma, \rho: [0, \infty)  \rightarrow [1, \infty)$ satisfying the following metric bounds
\begin{equation}\label{a3}
     \sigma(t)^{-1} \widetilde{\omega}(t) \leqslant \omega(t) \leqslant \sigma(t) \widetilde{\omega}(t),
\end{equation}
curvature bounds 
\begin{equation}\label{a4}
        \sup_{X}|Rm(\widetilde{\omega}(t))|_{\widetilde{\omega}(t)} \leqslant \tau(t),
    \end{equation}
and curvature derivative bounds 
\begin{equation}\label{a5}
        \sup_{X}|\widetilde{\nabla} Rm(\widetilde{\omega}(t))|_{\widetilde{\omega}(t)} \leqslant \rho(t).
    \end{equation}
Furthermore, assume that there exists a constant $B>0$, independent of time, such that for all $t \geqslant 1$
\begin{equation}\label{a2}
    \frac{1}{t}\tr_{\widetilde{\omega}(t)} \widetilde{\omega}_0  \geqslant B^{-1} >0.
\end{equation}
Then for any other initial metric $\omega_0$ with solution $\omega(t)$ to \eqref{cont}, there exists $C>0$ uniform of $t$ such that  
\begin{equation}\label{mainineq2}
    \sup_{X} |Rm(\omega(t))|_{\omega(t)} \leqslant C t^2\sigma^4 \tau^{3} \rho^{2}
\end{equation}
for all $t \geqslant 1$.
\end{thm}

\begin{rem}\label{rem1}
In the case where \eqref{a2} is a holomorphic submersion, the assumption \eqref{a2} is automatically satisfied. In a note by W. Jian and Y. Zhang, the authors proved that a  ``twisted" scalar curvature converges to $-kod(X)$ away from singular fibers. As a consequence, one can show \eqref{a2} on compact sets away from singular fibers, thus \eqref{a2} holds globally if \eqref{fmap} is a holomorphic submersion. 
\end{rem}

\begin{rem}
The assumption \eqref{a5} is serves as a substitute for Shi's estimate in the K\"ahler-Ricci Flow. To the author's knowledge, there is no such analogous result for the continuity method.
\end{rem}

The above theorems are particularly useful when considering product manifolds. Product manifolds serve as a prototype in understanding how geometry collapses under evolution equations such as the K\"ahler-Ricci flow and \eqref{cont}. In the volume collapsed case, the manifold $X$ is split into a collapsing component, $Y$, and base component, $B$, so that $X = Y \times B$. In such cases, a product metric solution to \eqref{cont} and the K\"ahler-Ricci flow can be easily constructed. On $B$, one endow a K\"ahler-Einstien metric, on $Y$, a Calabi-Yau metric. The behaviour regarding how curvature blows up for this product metric can be directly computed. However, the geometric behaviour is not well understood for  non-product metrics. Theorem \ref{thm1} and Theorem \ref{thm2} allows us to obtain curvature estimates for general solutions to \eqref{cont} on product manifolds.\\

We end the introduction with a brief description of the organization of the paper. In Section \ref{s2}, we recall various result from the literature on the continuity equation \eqref{cont}. In Section \ref{s3}, we prove the metric equivalence result stated in Theorem \ref{thm1}. We begin Section \ref{s4} by obtaining $C^3$ estimates for the solution metric $\omega(t)$ in terms of $\sigma(t), \tau(t),\rho(t)$ - the bounds for $\widetilde{\omega}(t)$. Using this, we can complete the proof of Theorem \ref{thm2}. The proofs are all maximum principle type arguments. Finally, we apply our theorems in product manifolds in Section \ref{s5}.

\subsection*{Acknowledgements}
I would like to thank my advisors, Zhou Zhang and Haotian Wu, for their constant encouragement, support and assistance in navigating through this fascinating field. I would like to thank Yashan Zhang for pointing out the result in Remark \ref{rem1} and for providing useful feedback. I would also like to thank Frederick Fong and Zhenlei Zhang for their helpful comments and interest in this work and the referee for their valuable feedback and suggestions.

\section{Background on the Continuity Equation}\label{s2}
To simplify our notation, we omit the $t$ dependence of the metrics; $\omega, \widetilde{\omega}$, and functions; $\sigma, \tau $ and $\rho$. We first recall some properties of the continuity equation. Let 
\begin{equation}
f: X \rightarrow X_{\text {can }} \subset \mathbb{C P}^{N}
\end{equation}
be the semi-ample fibration induced by the pluricanonical system. The target manifold $X_{can}$ is a $kod(X)$-dimensional projective variety and is called the canonical model of $X$. In this setting one obtains $\chi$, a multiple of the Fubini-Study metric on $\CC \PP^N$, such that its pullback $f^* \chi$ is a smooth semi-positive representation of $-2 \pi c_1(X)$. From a geometric perspective,  $2 \pi c_1(X)$ is the class which  contains the Ricci form of $X$.  Furthermore, we set $\Omega$ to be a smooth volume form on $X$ with 
$ \sqrt{-1} \partial \bar{\partial} \log \Omega=f^{*} \chi. $ The general strategy for studying \eqref{cont} is to use Cohomology to reduce the continuity equation to a Monge-Amp\`{e}re equation. More precisely, for $t \in [0, \infty)$ we define a family of reference metrics given by 
\begin{equation}\label{ref}
    \omega_t := \frac{1}{1+t} \omega_0 + \frac{t}{1+t}f^* \chi.
\end{equation}
This allows us to define $\varphi : X \times [0, \infty) \rightarrow \RR$ through 
\begin{equation}\label{defpot}
     \sqrt{-1}\p \bp \varphi = \omega - \omega_t  . 
\end{equation}
By substituting this expression into equation \eqref{cont}, we obtain the following elliptic Monge-Amp\`ere equation for $\varphi(t)$;
\begin{equation}\label{MA}
    \omega^n  = (1+t)^{-r}e^{\left( \frac{1+t}{t}\right) \varphi} \Omega.
\end{equation}
Here $n = \dim(X)$ and $r = n - kod(X)$, the dimension of a generic collapsing fiber. From the analysis carried out in \cite{ZZ19}, we have the following 
\begin{lem}[Lemma 2.1 in \cite{ZZ19}]\label{lem0}
    Let $\varphi(t)$ be the K\"ahler potential of $\omega(t)$. Then there exists $C>0$ such that 
    \begin{equation}
        |\varphi(t)| \leqslant C, \quad \text{for all }t \in [0,\infty).
    \end{equation}
\end{lem}
Although the result in \cite{ZZ19} are for elliptic surfaces,  $n=2$ and $kod(X)=1$, one can check that the arguments carry through for higher dimensions. As mentioned in \cite{ZZ19}, the proof follows from the pluripotential theory of degenerate complex Monge-Amp\`{e}re equation (\cites{DP10,EGZ08}).

Using Lemma \ref{lem0}, we can show the following volume equivalence which will be useful in the next section. 
\begin{lem}
Let $\omega(t)$ and $\widetilde{\omega}(t)$ be solutions to the continuity method starting from different initial metrics.  Then there exists a uniform $C>0$ such that 
\begin{equation}\label{lem2}
  C^{-1}\widetilde\omega^n  \leqslant \omega^n \leqslant C\widetilde\omega^n
\end{equation}
for all $t \geqslant 1$. 
\end{lem}

\begin{proof}
From the Monge-Amp\`{e}re equation \eqref{MA} we have 
$$\frac{1+t}{t} \varphi = \log\left( \frac{(1+t)^r \omega^n}{\Omega} \right) \quad \text{and} \quad \frac{1+t}{t} \widetilde{\varphi} = \log\left( \frac{(1+t)^r \widetilde{\omega}^n}{\Omega} \right).  $$
Combining the two equations above, we have 
$$ 
\frac{1+t}{t} (\varphi - \widetilde{\varphi}) = \log\left(\frac{\omega^n}{\widetilde{\omega}^n} \right).
$$
Then using Lemma \ref{lem0} in the above, we obtain \eqref{lem2}.
\end{proof}

Currently, is not known whether the scalar curvature is uniformly bounded along the continuity method. On compact sets away from singular fibers, a twisted scalar curvature converges to $-kod(X)$ by a result due to Jian and Y. Zhang. This is substantially different from the flow, where it was shown in \cite{J20} that the scalar curvature converges to $-kod(X)$ on compact sets away from singular fibers. Furthermore, global scalar curvature bounds for the K\"ahler-Ricci flow are known, see \cite{ST17}.

If we assume that our initial metric has negative curvature, then one can obtain a bound on scalar curvature. By direct calculation, one can show 
\begin{equation}\label{EvTr1}
    \Delta_{\omega} \tr_{\omega} \omega_0 \geqslant  - g^{\bar{j}i}g^{\bar{l}k}R^0_{i \bar{j} k \bar{l}}  +  \frac{1}{t}|\omega_0|_g^2 -  \frac{1+t}{t}  \tr_{\omega} \omega_0 + \frac{|\partial \tr_{\omega} \omega_0|_g^2}{\tr_{\omega} \omega_0}.
\end{equation}
Suppose that the curvature of the initial metric satisfies 
$$ R_{i\bar{i}j\bar{j}} \leqslant -A <0,$$
for some $C>0$. Then from \eqref{EvTr1}, we have  
$$
  \Delta_{\omega} \tr_{\omega} \omega_0 \geqslant A( \tr_{\omega} \omega_0)^2 - \frac{1+t}{t}  \tr_{\omega} \omega_0.
$$
Then applying a maximum principle argument, one obtains
\begin{equation}
    \tr_{\omega} \omega_0 \leqslant C,
\end{equation}
for $t$ away from $0$. For times close to $0$, we can simply use that $\omega$ is close to $\omega_0$ and thus the above trace quantity is bounded. In fact, this shows that the solution is volume non-collapsed. This reflects the coarse correspondence between Kodaria dimension and curvature; a negative Kodaria dimension corresponds to positive curvature, zero Kodaria dimension corresponds to flatness, and maximum Kodaria dimension (general type) corresponds to negative curvature.

\section{Metric Equivalence}\label{s3}

In this section, we give a proof for Theorem \ref{thm1} and derive Ricci curvature estimates as stated in Corollary \ref{cor1}. 

\begin{proof}[(Proof of Theorem \ref{thm1})]
For convenience of notation, we drop the time dependence for the metrics and denote $C>0$ to be some constant which may change in each line but can be fixed at the end of each proof. We also denote the Hermitian metric associated to $\omega$ by $g$ and apply the same convention for $\tilde{\omega}$ with $\tilde{g}$ and $\omega_0$ with $g_0$. Recall the standard identity 
$$
    \Delta_\omega \tr_{\omega} \widetilde{\omega} = g^{\bar{b}p}g^{\bar{q}a}\widetilde{g}_{p \bar{q}}R_{a\bar{b}} - g^{\bar{j}i}
    g^{\bar{q}p}\widetilde{R}_{i \bar{j}p \bar{q}} + g^{\bar{j}i}g^{\bar{q}p} \widetilde{g}^{ \bar{b}a} \nabla_i \widetilde{g}_{p \bar{b}} \nabla_{\bar{j}} \widetilde{g}_{a \bar{q}}.
$$
The first term can be rewritten in terms of metrics using \eqref{cont} and the second term contains a controlled curvature term due to \eqref{a1}. Then  using the Cauchy-Schawrz inequality, we have 
\begin{equation}\label{eqn1}
    \Delta_\omega \tr_{\omega} \widetilde{\omega} \geqslant - \frac{C}{t} \tr_\omega \omega_0 \tr_\omega \widetilde{\omega} - \frac{1+t}{t} \tr_\omega \widetilde{\omega} - C \tau  (\tr_{\omega} \widetilde{\omega})^2 + g^{\bar{j}i}g^{\bar{q}p} \widetilde{g}^{ \bar{b}a} \nabla_i \widetilde{g}_{p \bar{b}} \nabla_{\bar{j}} \widetilde{g}_{a \bar{q}},
\end{equation}
for all $t \geqslant 1$. Since \eqref{a1} yields a Ricci curvature bound for $\widetilde{\omega}$ which satisfies \eqref{cont}, we have 
$$ \omega_0 \leqslant C \widetilde{\omega}_0 \leqslant C t \tau \widetilde{\omega},$$
for some $C>0$ and for all $t \geqslant 1$. Substituting the above inequality into \eqref{eqn1} yields 
\begin{equation}\label{trev}
 \Delta_\omega \tr_{\omega} \widetilde{\omega} \geqslant  - C \tau (\tr_{\omega} \widetilde{\omega})^2 -C  \tr_{\omega} \widetilde{\omega} +  g^{\bar{j}i}g^{\bar{q}p} \widetilde{g}^{ \bar{b}a} \nabla_i \widetilde{g}_{p \bar{b}} \nabla_{\bar{j}} \widetilde{g}_{a \bar{q}}.
\end{equation}
Then by a standard trick using the Cauchy-Schwarz inequality, we can rewrite \eqref{trev} as 
\begin{equation}
    \Delta_{\omega} \log \tr_{\omega} \widetilde{\omega} \geqslant -C - C \tau \tr_{\omega} \widetilde{\omega}.
\end{equation}
We divide through by $\tau$, noting that $\tau \geqslant 1$, to obtain 
\begin{equation}\label{maxp1}
    \Delta_{\omega}(\tau^{-1} \log \tr_{\omega}\widetilde{\omega}) \geqslant - C\tr_{\omega} \widetilde{\omega} -C. 
\end{equation}
On the other hand, taking the trace of \eqref{defpot} for both $\widetilde{\omega}$  and $\omega$ respectively yields 
\begin{equation}\label{phievoltilde}
    \Delta_{\omega} \widetilde{\varphi} = \tr_{\omega} \widetilde{\omega} - \tr_{\omega} \widetilde{\omega}_t,
\end{equation}
and 
\begin{equation}\label{phievol}
        \Delta_{\omega} \varphi = n - \tr_{\omega}\omega_t,
\end{equation}
respectively. Here we define $\widetilde{\omega}_t$ to be the reference metric similarly defined as in \eqref{ref} but with $\widetilde{\omega}_0$ in place of $\omega_0$. Let $A \geqslant 1$ be large enough so that $A \omega_0 > \widetilde{\omega}_0$. From \eqref{phievoltilde} and \eqref{phievol}, we have 
$$ \Delta_\omega (\widetilde{\varphi} - A \varphi) = \tr_\omega \widetilde{\omega} - An + \frac{1}{1+t} (A \tr_\omega \omega_0 - \tr_\omega \widetilde{\omega}_0) + (A-1) \frac{t}{1+t}\tr_\omega f^* \chi.$$
This gives us the following inequality 
\begin{equation}\label{maxp2}
    \Delta_\omega (\widetilde{\varphi} - A \varphi) \geqslant \tr_\omega \widetilde{\omega} - C.
\end{equation}
Now, for a large enough $B>0$, using \eqref{maxp1} and \eqref{maxp2}, we derive the estimate
\begin{equation*}
     \Delta_{\omega} \left( \tau^{-1} \log \tr_{\omega} \widetilde{\omega} + B(\widetilde{\varphi} - A \varphi \right) ) \geqslant  C \tr_{\omega} \widetilde{\omega} - C. 
\end{equation*}
Then applying a standard maximum principle argument   with the help of the uniform bounds in Lemma \ref{lem0}, we obtain a uniform $C>0$ such that 
\begin{equation}\label{tr1}
    \tr_{\omega} \widetilde{\omega} \leqslant Ce^{C \tau}
\end{equation} 
holds on $X$ for all $t \geqslant 1$. This yields the lower bound in \eqref{meq}. To deduce the upper estimate, we use 
$$ \tr_{\widetilde{\omega}} \omega \leqslant n (\tr_{\omega} \widetilde{\omega})^{n-1}  \frac{\omega^n}{\widetilde{\omega}^n}$$
and Lemma \ref{lem2} to show   
\begin{equation}\label{tr2}
    \tr_{\widetilde{\omega}} \omega \leqslant Ce^{C \tau},
\end{equation} 
for all $t \geqslant 1$. The upper estimate immediately follows from the above inequality. 
\end{proof}

From the metric equivalence in \ref{thm1}, one can immediately obtain Ricci curvature estimates along \eqref{cont}. 

\begin{proof}[(Proof of Corollary \ref{cor1})]
A lower bound for the Ricci curvature always holds for the continuity equation. Indeed,
\begin{equation}
        \Ric(\omega(t)) = \frac{1}{t} \omega_0 - \frac{1+t}{t} \omega(t) \geqslant - 2 \omega(t)
\end{equation}
for all $t \geqslant 1$. To obtain the upper bound, we choose $A>0$ large enough such that $\omega_0 \leqslant A \widetilde{\omega}_0$ and rewrite 
\begin{equation}
\begin{aligned}
            \Ric(\omega(t)) &= \frac{1}{t} \omega_0 - \frac{1+t}{t} \omega(t) \\
            & \leqslant \frac{A}{t} \widetilde{\omega}_0 \\
            & \leqslant C \tau \widetilde{\omega} \\
            & \leqslant C \tau e^{C \tau} \omega. 
\end{aligned}
\end{equation}
In the second last line, we used \eqref{a1} to deduce that $\widetilde{\omega}_0 \leqslant Ct\tau \widetilde{\omega}$ from the scalar curvature bound. The last line follows from \eqref{tr1}. From the Ricci bound, the scalar curvature bounds immediately follow. 
\end{proof}


To conclude this section, we make note of an estimate which will be useful throughout the next section. Suppose we are in the setup of Theorem \ref{thm2}. Taking the trace of \eqref{cont}, the scalar curvature $R(t)$ is given by
\begin{equation}\label{R}
     R(t) = \frac{1}{t} \tr_\omega \omega_0  - n \frac{1+t}{t}.
\end{equation}
An upper bound for scalar curvature implies that there exists $C>0$ such that 
\begin{equation}\label{tr0bdd}
     \tr_{\omega} \omega_0 \leqslant C t \sigma \tau,
\end{equation}
for all $t \geqslant 1$.

\section{Curvature Estimates}\label{s4}

We first obtain Calabi $C^3$-bounds for the metric $\omega$ in terms of the bounds for $\widetilde{\omega}$. This will be crucial when  applying the maximum principle to obtain the curvature bound in Theorem \ref{thm2}. Define the tensor $\Psi_{i \bar{j}}^k := \Gamma_{i \bar{j}}^k - \widetilde{\Gamma}_{i \bar{j}}^k$ and its norm $S = |\Psi|_{\omega(t)}^2$. In local coordinates,
\begin{equation}\label{defS}
    S = g^{\bar{j}i}g^{\bar{l}k}g^{\bar{q}p}\widetilde{\nabla}_i g_{k \bar{q}} \widetilde{\nabla}_{\bar{j}} g_{p \bar{l}}.
\end{equation}
If $S$ and $T$ are tensors, we write $S * T$ for any linear combination of the products of tensors $S$ and $T$ formed by contractions using the metric $g$. 

\begin{lem}\label{lem3}
Let $S$ be defined as above, the norm difference between christoffel symbols for the evolving metric $\omega$ and a given metric $\tilde{\omega}$. We denote their corresponding geometric quantities similarly; $\tilde{\omega}$, $\tilde{g}$, $\widetilde{\Ric}$, $\widetilde{\Rm}$ and so on. Then 
\begin{equation}\label{lem3eqn}
\begin{aligned}
        \Delta_\omega S \geqslant  & -C \sigma \tau S -  C|\Psi * R m(\widetilde{\omega})|_{\omega} \sqrt{S}  -C \left|g^{\bar{b} a} \widetilde{\nabla}_{a} \widetilde{R}_{i \bar{b} p}{ }^{k}\right|_{\omega}\sqrt{S} \\
        & - 2\sigma \tau \sqrt{S_0} \sqrt{S}  +  |\nabla \Psi|_{\omega}^{2}+|\bar{\nabla}\Psi|_{\omega}^{2}.
\end{aligned}
\end{equation}
\end{lem}

\begin{proof}
The laplacian of $S$ is given by 
\begin{equation}\label{Sexp}
\begin{aligned}
\Delta_{\omega} S=& R^{\bar{j} i} g^{\bar{q} p} g_{k \bar{l}} \Psi_{i p}^{k} \overline{\Psi_{j q}^{l}}+g^{\bar{j} i} R^{\bar{q} p} g_{k \bar{l}} \Psi_{i p}^{k} \overline{\Psi_{j q}^{l}}-g^{\bar{j} i} g^{\bar{q} p} R_{k \bar{l}} \Psi_{i p}^{k} \overline{\Psi_{j q}^{l}}\\
& + 2 \operatorname{Re}\left(g^{\bar{j} i} g^{\bar{q} p} g_{k \bar{l}}\left(\Delta_{\omega} \Psi_{i p}^{k}\right) \overline{\Psi_{j q}^{l}}\right) \\
&+|\nabla \Psi|_{\omega}^{2}+|\bar{\nabla} \Psi|_{\omega}^{2} \\
& =:  I + II +  |\nabla \Psi|_{\omega}^{2}+|\bar{\nabla}\Psi|_{\omega}^{2}. 
\end{aligned}
\end{equation}
Here, we denote the terms in the first line of \eqref{Sexp} by $I$ and second line of \eqref{Sexp} by $II$. In the K\"ahler-Ricci flow, the Ricci terms in $I$ cancel with the time evolution. For the continuity method, we instead use \eqref{cont} to rewrite the Ricci terms in terms of $\omega$ and $\omega_0$ for $t \geqslant 1$;
\begin{equation}
    \begin{aligned}
        I = \frac{1}{t}\left(g_0^{\bar{j}i} g^{\bar{q}p} g_{k \bar{l}} + g^{\bar{j}i}g^{\bar{q}p}_0 g_{k \bar{l}}- g^{\bar{j}i}g^{\bar{q}p}g_{k\bar{l}}^0 \right) \Psi^k_{ip} \overline{\Psi_{jq}^l} - \frac{1+t}{t}S.
    \end{aligned}
\end{equation}
We apply local normal coordinates which diagonalizes $\omega_0$ with respect to $\omega$. Let $\lambda^0_i$ be the principle eigenvalues of $g_0$. Using the upper trace bound \eqref{tr0bdd} in local coordinates, $ \lambda_i^0 \leqslant C t \sigma \tau$,  we obtain the estimate   
\begin{equation}\label{IIineq}
    \begin{aligned}
    I &= \frac{2}{t \lambda_i^0} \Psi_{ip}^k \overline{\Psi_{ip}^k} - \frac{\lambda^0_i}{t} \Psi_{ip}^k \overline{\Psi_{ip}^k} - \frac{1+t}{t} S\\
    & \geqslant \frac{2}{C \sigma \tau t^2} S - C\sigma \tau S - \frac{1+t}{t}S \\
    &\geqslant -C \sigma \tau S 
\end{aligned}
\end{equation}
for all $t \geqslant 1$. Moving onto the terms in $II$, we can use  
\begin{equation}
\Delta_{\omega} \Psi_{i p}^{k}=\nabla^{\bar{b}} \widetilde{R}_{i \bar{b} p}^{k}-\nabla_{i} R_{p}^{k},
\end{equation}
to rewrite $II$ as 
\begin{equation}\label{Iexpr}
    II = 2 \operatorname{Re}\left(g^{\bar{j} i} g^{\bar{q} p} g_{k \bar{l}}\left(\nabla^{\bar{b}} \widetilde{R}_{i \bar{b} p}^{k} \right) \overline{\Psi_{j q}^{l}}\right) -
2 \operatorname{Re}\left(g^{\bar{j} i} g^{\bar{q} p} g_{k \bar{l}}\left(\nabla_{i} R_{p}^{k}\right) \overline{\Psi_{j q}^{l}}\right).
\end{equation}
To treat the first term in \eqref{Iexpr}, we can use 
\begin{equation}\label{orange}
    \begin{aligned}
    \nabla^{\bar{b}} \widetilde{R}_{i \bar{b} p}^{k} &=g^{\bar{b} a} \nabla_{a} \widetilde{R}_{i \bar{b} p}^{k} \\
    &=g^{\bar{b} a}\left(\nabla_{a}-\widetilde{\nabla}_{a}\right) \widetilde{R}_{i \bar{b} p}^{k}+g^{\bar{b} a} \widetilde{\nabla}_{a} \widetilde{R}_{i \bar{b} p}^{k} \\
    &=\Psi * R m(\widetilde{\omega})+g^{\bar{b} a} \widetilde{\nabla}_{a} \widetilde{R}_{i \bar{b} p}^{k}.
    \end{aligned}   
\end{equation}
For the second term in \eqref{Iexpr}, we once again use the continuity equation \eqref{cont} to obtain  
$$R_p^k = \frac{1}{t}g^{\bar{r}k}g^0_{p \bar{r}} -  \frac{1+t}{t} g^{\bar{r}k}g_{p \bar{r}}. $$
Here, $g_0$ is the metric associated with the form $\omega_0$.  Taking the covariant derivative and using metric compatibility yields
$$\nabla_i R_p^k =  \frac{1}{t}g^{\bar{r}k} \nabla_i g^0_{p \bar{r}}. $$
Hence, we can rewrite said second term of $II$ as
\begin{equation}\label{note2}
\begin{aligned}
    -2 \operatorname{Re}\left(g^{\bar{j} i} g^{\bar{q} p} g_{k \bar{l}}\left(\nabla_{i} R_{p}^{k}\right) \overline{\Psi_{j q}^{l}}\right) &=  -\frac{2}{t} \operatorname{Re}\left(g^{\bar{j} i} g^{\bar{q} p}   \left(\nabla_i g^0_{p \bar{l}} \right) \overline{\Psi_{j q}^{l}}\right) \\
    & = -\frac{2}{t} \operatorname{Re}\left(g^{\bar{j} i} g^{\bar{q} p} \left( (\nabla_i - \nabla_i^0) g^0_{p \bar{l}}\right) \overline{\Psi_{j q}^{l}}\right) \\ 
    & =  \frac{2}{t} \operatorname{Re}\left(g^{\bar{j} i} g^{\bar{q} p} g_{n \bar{l}}^0 \mathring{\Psi}_{ip}^{n}  \overline{\Psi_{j q}^{l}}\right) \\
    & \geqslant - 2\sigma \tau \sqrt{S_0} \sqrt{S}
\end{aligned}
\end{equation}
where $\Psi^0$  and $S_0$ are defined by 
$$(\Psi^0)_{i \bar{j}}^k := \Gamma_{i \bar{j}}^k - (\Gamma_0)_{i \bar{j}}^k \text{ and } S_0 = |\mathring{\Psi}|_{\omega}^2.$$

Using Cauchy Schwarz inequality and the above, we can estimate $II$ from below by 
\begin{equation}\label{IIfin}
        II \geqslant -C|\Psi * R m(\widetilde{\omega})|_{\omega} \sqrt{S}  -C \left|g^{\bar{b} a} \widetilde{\nabla}_{a} \widetilde{R}_{i \bar{b} p}{ }^{k}\right|_{\omega} \sqrt{S} - 2\sigma \tau \sqrt{S_0} \sqrt{S}.
\end{equation}
Thus, combining \eqref{Sexp}, \eqref{IIineq} and \eqref{IIfin}, we have \eqref{lem3eqn}.
\end{proof}

Using Lemma \ref{lem3}, we obtain a bound for $S$.
\begin{prop}\label{Sbdd}
Suppose we are in the setting of Theorem \ref{thm2} and let $S$ be defined as in \eqref{defS}. There exists $C>0$ such that
\begin{equation}\label{eqnSbdd}
    S \leqslant C t^2 \sigma^4 \tau^3 \rho^2,
\end{equation}
for all $t \geqslant 1$. 
\end{prop}

\begin{proof}
We begin by observing that the last term in \eqref{trev} can be estimated by 
\begin{equation}
\begin{aligned}
g^{\bar{j} i} g^{\bar{q} p} \widetilde{g}^{\bar{b} a} \nabla_{i} \widetilde{g}_{p \bar{b}} \nabla_{\bar{j}} \widetilde{g}_{a \bar{q}} &=g^{\bar{j} i} g^{\bar{q} p} \widetilde{g}^{\bar{b} a}\left(\nabla_{i}-\widetilde{\nabla}_{i}\right) \widetilde{g}_{p \bar{b}}\left(\nabla_{\bar{j}}-\widetilde{\nabla}_{\bar{j}}\right) \widetilde{g}_{a \bar{q}} \\
&=g^{\bar{j} i} g^{\bar{q} p} \widetilde{g}^{\bar{b} a}\left(-\Psi_{i p}^{d}\right) \widetilde{g}_{d \bar{b}}\left(-\overline{\Psi_{j q}^{e}}\right) \widetilde{g}_{a \bar{e}} \\
& \geqslant \sigma^{-1} S.
\end{aligned}
\end{equation}
By assumption \eqref{a3}, the remaining terms in \eqref{trev} are controlled. Therefore, 
\begin{equation}\label{trwtildew}
    \Delta_{\omega}(\sigma^{-1} \tr_{\omega} \widetilde{\omega}) \geqslant -C - C\tau \sigma +\sigma^{-2}S.
\end{equation}
On the other hand, \eqref{a4}, \eqref{a5} and  \eqref{a2}, gives us 
\begin{equation}\label{note1}
\begin{aligned}
|\Psi * R m(\widetilde{\omega})|_{\omega} & \leqslant |\Psi|_{\omega}|R m(\widetilde{\omega})|_{\omega} \leqslant C \sigma^{2} \tau|\Psi|_{\omega} \quad \text{and} \\
\left|g^{\bar{b} a} \widetilde{\nabla}_{a} \widetilde{R}_{i \bar{b} p}{ }^{k}\right|_{\omega} & \leqslant \sigma\left|\widetilde{\nabla}^{\bar{b}} \widetilde{R}_{i \bar{b} p}{ }^{k}\right|_{\omega} \leqslant \sigma^{\frac{5}{2}}\left|\widetilde{\nabla}^{\bar{b}} \widetilde{R}_{i \bar{b} p}{ }^{k}\right|_{\widetilde{\omega}} \leqslant C \sigma^{\frac{5}{2}} \rho.\\
\end{aligned}
\end{equation}
Thus, we are left to bound the $S_0$ term in \eqref{lem3eqn}. To this end, we make the following claim.\\

\textbf{Claim:} There exists a uniform $C>0$ such that 
\begin{equation}\label{bddS0}
    S_0 \leqslant C t^2 \sigma^4 \tau^3 
\end{equation}
for all $t \geqslant 1$. \\

We assume this for now and proceed with the proof of Proposition \ref{Sbdd}.  Combining  \eqref{note1} and \eqref{bddS0} with \eqref{lem3eqn}, we obtain 
$$
    \Delta_\omega  S \geqslant -C \sigma \tau S -  C\sigma^2 \tau S  -C \sigma^\frac{5}{2} \rho \sqrt{S} 
    - 2Ct \sigma^3 \tau^\frac{5}{2} \sqrt{S} +|\nabla \Psi|_{\omega}^{2}+|\bar{\nabla} \Psi|_{\omega}^{2}
$$
for all $t \geqslant 1$. After collecting terms by their power of $S$, using that $\sigma, \tau,\rho, t \geqslant 1$, and discarding the norms of $\nabla \Psi$, the above inequality simplifies to 
$$
    \Delta_\omega S \geqslant -C\sigma^2 \tau S - C t \sigma^{3}  \tau^\frac{5}{2}\rho \sqrt{S}.
$$
In order to apply the maximum principle, we rewrite this as
\begin{equation}\label{Sevol}
    \Delta_\omega (\sigma^{-4} \tau^{-1} S) \geqslant  -C\sigma^{-2} S - C t \sigma^{-1} \tau^{\frac{3}{2}} \rho \sqrt{S}.
\end{equation}
Let $Q := \sigma^{-4} \tau^{-1} S + A \sigma^{-1} \tr_{\omega} \widetilde{\omega}$ for some arbitrary $A>0$. From \eqref{trwtildew} and \eqref{Sevol}, we have 
\begin{equation}
    \Delta_\omega Q \geqslant  (A- C) \sigma^{-2}S - Ct \sigma^{-1} \tau^{\frac{3}{2}} \rho \sqrt{S} - AC - A \tau \sigma,
\end{equation}
for some constant $C>0$ and all $t \geqslant 1$. We can always choose $A \geqslant 2C$  large enough to force a sign on the coefficients of $S$. Moreover, we can assume that at the maximum of point of $Q$,  
$$  t \sigma^{- 1} \tau^{\frac{3}{2}} \rho \sqrt{S} \leqslant \frac{1}{2}\sigma^{-2} S.$$
Otherwise, we would have $ S \leqslant C t^2 \sigma^2 \tau^3 \rho^2$ at this point, which would lead to \eqref{eqnSbdd}. Putting all this together, we obtain a lower estimate at the maximum of $Q$:
\begin{equation}
        \Delta_\omega Q \geqslant \frac{C}{2} \sigma^{-2}S - A\tau \sigma - AC.
\end{equation}
Then, by the maximum principle, we obtain the estimate
\begin{equation}
    S \leqslant C \sigma^4 \tau \leqslant C t^2 \sigma^4 \tau^3 \rho^2
\end{equation}
on all of $X$ and for all $t \geqslant 1$.

We now prove the claim. From a similar calculation leading to \eqref{trev}, we can find some $C>0$ such that 
\begin{equation}\label{eqnevoltr0}
     \Delta_\omega \left( \frac{1}{t} \tr_{\omega} \omega_0 \right) \geqslant - \frac{C}{t} -C \sigma \tau  -C t \sigma^2 \tau^2 + \frac{1}{t} g^{\bar{j}i} g^{\bar{q}p}g_{a \bar{b}}^0 \mathring{\Psi}_{ip}^d \overline{\mathring{\Psi}_{jq}^e},
\end{equation}
for all $t \geqslant 1$. To estimate the last term, we use \eqref{a3} and \eqref{a2} to obtain 
$$t^{-1} \tr_{\omega} \omega_0 \geqslant \sigma^{-1} t^{-1} \tr_{\widetilde{\omega}} \widetilde{\omega}_0 > C^{-1}\sigma^{-1}. $$
Applying normal coordinates for which $\omega_0$ is diagonalised with respect to $\omega$, we obtain  
\begin{equation}
        \frac{1}{t} g^{\bar{j}i} g^{\bar{q}p}g_{a \bar{b}}^0 \mathring{\Psi}_{ip}^d \overline{\mathring{\Psi}_{jq}^e} = \frac{1}{t}\lambda_a^0 S_0 \geqslant C^{-1} \sigma^{-1} S_0.
\end{equation}
As a consequence, we have for some $C>0$, 
\begin{equation}\label{trw0evol}
    \Delta_\omega \left( t^{-1} \tr_{\omega} \omega_0 \right) \geqslant - C t^{-1} -C \sigma \tau  -C t \sigma^2 \tau^2 + C^{-1} \sigma^{-1} S_0,
\end{equation}
for all $t \geqslant 1$. We can also compute a lower bound for $\Delta_\omega S_0$ using Lemma \ref{lem3}. Here we have $\mathring{\Psi}$ instead of $\Psi$ and $\Rm^0$ in place of $\widetilde{Rm}$. Moreover, by \eqref{tr0bdd} we have 
$$ \begin{aligned}
|\mathring{\Psi} * R m(\omega_0)|_{\omega} & \leqslant |\mathring{\Psi}|_{\omega}|R m(\omega_0)|_{\omega} \leqslant C t^2 \sigma^{2} \tau^2 |\mathring{\Psi}|_{\omega} \quad \text{and} \\
\left|g^{\bar{b} a} \mathring{\nabla}_{a} \mathring{R}_{i \bar{b} p}{ }^{k}\right|_{\omega} & \leqslant t \sigma \tau \left|\mathring{\nabla}^{\bar{b}} \mathring{R}_{i \bar{b} p}{ }^{k}\right|_{\omega} \leqslant  (t \sigma \tau)^2\left|\mathring{\nabla}^{\bar{b}} \mathring{R}_{i \bar{b} p}{ }^{k}\right|_{\omega_0} \leqslant C t^2 \sigma^2 \tau^2.\\
\end{aligned} $$
Therefore, for some $C>0$, we have  
\begin{equation}\label{S0evol}
    \Delta_\omega S_0 \geqslant - C t^2 \sigma^2 \tau^2 S_0 -C t^2 \sigma^2 \tau^2 \sqrt{S_0} + |\nabla \mathring{\Psi}|^2_\omega + |\overline{\nabla} \mathring{\Psi}|_\omega^2,
\end{equation}
for all $t \geqslant 1$. For some $A>0$, we set 
$$Q:= t^{-2} \sigma^{-4} \tau^{-3} S_0 + At^{-1} \sigma^{-1} \tau^{-1} \tr_\omega \omega_0. $$
We can then combine \eqref{trw0evol} and \eqref{S0evol} to estimate the Laplacian of Q from below. Then choosing $A$ large enough, we can force a sign on the coefficient of $S_0$ to obtain 
\begin{equation}
    \Delta_\omega Q \geqslant C \sigma^{-2} \tau^{-1} S_0 - C \sigma^{-2} \tau^{-1} \sqrt{S_0}- Ct \sigma \tau, 
\end{equation}
for some $C>0$ and for all $t \geqslant 1$. We assume that $S_0 \geqslant 1$ and apply the maximum principle to show that at a maximum of $Q$,
$$ S_0 \leqslant Ct \sigma^3 \tau^2 .$$
This bound and the uniform bound of $t^{-1} \sigma^{-1} \tau^{-1} \tr_\omega \omega_0$ establishes \eqref{bddS0}. 
\end{proof}

With the proof of Proposition \ref{Sbdd} completed, we can proceed to prove Theorem \ref{thm2}. 

\begin{proof}[Proof of Theorem \ref{thm2}]

From Lemma 3.2.10 in \cite{SW13}, we have 
\begin{equation}
\Delta_{\omega}  \Rm = \nabla_{\bar{l}} \nabla_{k} R_{i \bar{j}}+  \mathrm{Rm} * \mathrm{Rm}+\mathrm{Rc} * \mathrm{Rm}. 
\end{equation}
Once again, we use \eqref{cont} to rewrite the Ricci term:
\begin{equation}
    \nabla_{\bar{l}} \nabla_{k} R_{i \bar{j}} = \frac{1}{t} \nabla_{\bar{l}} \nabla_{k}g_{i \bar{j}}^0 - \frac{1+t}{t} \nabla_{\bar{l}} \nabla_{k} g_{i \bar{j}}.
\end{equation}
The second term is zero by metric compatibility. The first term, expressed in normal coordinates with respect to $g$, can be expressed as 
$$
\begin{aligned}
    \frac{1}{t} \nabla_{\bar{l}} \nabla_{k} g_{i \bar{j}}^0 &= \frac{1}{t} \p_{\bar{l}} \p_k g_{i \bar{j}}^0 - \frac{1}{t} R_{\overline{j}i \overline{l}}^{\overline{m}} g_{k \bar{m}}^0, \\ 
    & = - \frac{1}{t} R_{k \bar{l} i \bar{j}}^0 + \frac{1}{t} g_0^{\bar{q}p}(\p_k g_{i \bar{q}}^0)(\p_{\bar{l}}g_{p \bar{j}}^0)-  \frac{1}{t} R_{\overline{j}i \overline{l}}^{\overline{m}} g_{k \bar{m}}^0, \\
    &= - \frac{1}{t}  R_{k \bar{l} i \bar{j}}^0 + \frac{1}{t} g_0^{\bar{q}p}(\nabla_k g_{i \bar{q}}^0)(\nabla_{\bar{l}}g_{p \bar{j}}^0)- \frac{1}{t} R_{\overline{j}i \overline{l}}^{\overline{m}} g_{k \bar{m}}^0,\\
    & = - \frac{1}{t}  R_{k \bar{l} i \bar{j}}^0 + \frac{1}{t} g^0_{n\bar{m}}\mathring{\Psi}_{ki}^n \overline{\mathring{\Psi}_{lj}^m} - \frac{1}{t} R_{\overline{j}i \overline{l}}^{\overline{m}} g_{k \bar{m}}^0.
\end{aligned}
$$
Thus, for any $t \geqslant 1$, we have in normal coordinates
\begin{equation}\label{Rmexpr}
    \Delta_\omega R_{i \bar{j}k\bar{l}} = - \frac{1}{t} Rm(\omega_0) + \frac{1}{t}g^0_{n\bar{m}}\mathring{\Psi}_{ki}^n \overline{\mathring{\Psi}_{lj}^m} - \frac{1}{t} R_{\overline{j}i \overline{l}}^{\overline{m}} g_{k \bar{m}}^0+ \Rm * \Rm + \Rm * \Ric.
\end{equation}
Denote $\Phi_{kilj} := t^{-1}g_{n \bar{m}}^0 \mathring{\Psi}_{ki}^n \overline{\mathring{\Psi}_{lj}^m}$. Then using the above and the metric trace bound \eqref{tr0bdd}, we have 
$$
\begin{aligned}
        \Delta_\omega \left(|\Rm(\omega)|_{\omega}^2  \right) \geqslant & - \frac{1}{t}|\Rm(\omega_0)|_{\omega} |\Rm(\omega)|_{\omega} - |\Phi|_{\omega} |\Rm(\omega)|_{\omega} - C \sigma \tau |\Rm(\omega)|_{\omega}^2 \\
        & - C|\Rm(\omega)|_{\omega}^3 + |\nabla \Rm(\omega)|_{\omega}^2 + |\overline{\nabla} \Rm(\omega)|_{\omega}^2.
\end{aligned}
$$
 In normal coordinates with respect to $\omega$, we can deduce an upper bound for $\Phi$:
$$ |\Phi|_{\omega} = t^{-1}|g_{n \bar{m}}^0 \mathring{\Psi}_{ki}^n \overline{\mathring{\Psi}_{lj}^m}|_\omega \leqslant C \sigma  \tau S_0 \leqslant C t^2 \sigma^5 \tau^4.$$
Using \eqref{tr0bdd}, we also have 
$$|\Rm(\omega_0)|_\omega \leqslant C t^2 \sigma^2 \tau^2|\Rm(\omega_0)|_{\omega_0} \leqslant Ct^2 \sigma^2 \tau^2. $$
Thus, taking $t \geqslant 1$, using \eqref{bddS0} and the previous two estimates, we deduce that for some uniform $C>0$, 
\begin{equation}\label{Rm2evol}
\begin{aligned}
        \Delta_\omega \left(|\Rm(\omega)|^2_\omega  \right) \geqslant & -Ct\sigma^2 \tau^2|\Rm(\omega)|_\omega - Ct^2 \sigma^5 \tau^4 |\Rm(\omega)|_\omega - C \sigma \tau |\Rm(\omega)|_\omega^2 \\
        &  - C |\Rm(\omega)|_\omega^3 + |\nabla \Rm(\omega)|_\omega^2 + |\overline{\nabla} \Rm(\omega)|_\omega^2 \\
         \geqslant &  - Ct^2 \sigma^5 \tau^4 |\Rm(\omega)|_\omega  - C \sigma \tau |\Rm(\omega)|_\omega^2  \\
        & - C |\Rm(\omega)|_\omega^3 + |\nabla \Rm(\omega)|_\omega^2 + |\overline{\nabla} \Rm(\omega)|_\omega^2
\end{aligned}
\end{equation}
In the last inequality, we collected terms by their power of curvature using that $\sigma, \tau,\rho, t \geqslant 1$. To simplify our notation, we set  
$$\widetilde{S} = t^{-2}\sigma^{-4}\tau^{-3}\rho^{-2}S, $$
which is uniformly bounded for $t \geqslant 1$ due to \eqref{eqnSbdd}. From the proof in Proposition \eqref{Sbdd}, there exists $C>0$ such that  
\begin{equation}\label{RmSevol}
    \Delta_{\omega}\widetilde{S} \geqslant - C \sigma^2 \tau  + t^{-2} \sigma^{-4} \tau^{-3} \rho^{-2} \left( |\nabla \Psi |_{\omega}^2 + |\overline{\nabla} \Psi |_\omega^2 \right),
\end{equation}
for all $t \geqslant 1$. 

Next, for $A>0$ large enough such that $A - \tilde{S}>1$, we define the quantity  
\begin{equation}
    Q := \frac{|\Rm(\omega)|_\omega^2}{A - \widetilde{S}}.
\end{equation}
 Since $\widetilde{S}$ is uniformly bounded, we can take $A$ large enough so that the $A - \widetilde{S}$ is positive. We will now set up a maximum principle argument to complete the proof for our result. Taking the laplacian and applying Cauchy-Schwarz, we can find some $C>0$ such that  
\begin{equation}
    \begin{aligned}
            \Delta_\omega  Q & \geqslant - C t^2 \sigma^5 \tau^4  \frac{|\Rm(\omega)|_\omega}{A- \widetilde{S}}  - C \sigma \tau \frac{|\Rm(\omega)|_\omega^2}{A- \widetilde{S}} - C \frac{|\Rm(\omega)|_\omega^3}{A - \widetilde{S}} + \frac{|\nabla \Rm(\omega)|_\omega^2}{A -\widetilde{S}} + \frac{|\overline{\nabla} \Rm(\omega)|_\omega^2}{A -\widetilde{S}} \\
            &  - C \sigma^2 \tau \frac{|\Rm(\omega)|_\omega^2}{(A- \widetilde{S})^2} + t^{-2}\sigma^{-4} \tau^{-3} \rho^{-2} (|\nabla \Psi|_\omega^2 + | \overline{\nabla} \Psi|_\omega^2) \frac{|\Rm(\omega)|_\omega^2}{(A- \widetilde{S})^2} \\
            & - 2 \frac{|\p |\Rm(\omega)|_\omega^2|_\omega|\p \widetilde{S}|_\omega}{(A- \widetilde{S})^2}   + 2 \frac{|\Rm(\omega)|_\omega^2|\p \widetilde{S}|_\omega}{(A- \widetilde{S})^3},
    \end{aligned}
\end{equation}
for all $t \geqslant 1$. The goal here is to produce a positive fourth power curvature term to dominate the negative curvature terms. This will come from the third-last term. To deal with the second-last term, we use
$$ |\p |\Rm(\omega)|_\omega^2|_\omega^2 \leqslant 2 |\Rm(\omega)|_\omega^2 (|\nabla \Rm(\omega)|_\omega^2 + |\overline{\nabla} \Rm(\omega)|_\omega^2 )$$
in combination with Young's inequality to obtain 
$$ 2 \frac{|\p |\Rm(\omega)|_\omega^2|_\omega|\p \widetilde{S}|_\omega}{(A- \widetilde{S})^2} \leqslant  4 \frac{|\Rm(\omega)|_\omega^2}{(A- \widetilde{S})^3}|\p \widetilde{S}|^2_\omega + \frac{|\nabla \Rm(\omega)|_\omega^2 + |\overline{\nabla} \Rm(\omega)|_\omega^2}{A - \widetilde{S}}.$$
Furthermore, by \eqref{eqnSbdd} we have
\begin{equation*}
    \begin{aligned}
            |\p \widetilde{S}|_\omega^2 &= t^{-4}\sigma^{-8}\tau^{-6} \rho^{-4}|\p S|_\omega^2 \\ & \leqslant 4t^{-4} \sigma^{-8} \tau^{-6} \rho^{-4}  S \left( |\nabla \Psi|_\omega^2 + |\overline{\nabla}
\Psi |_\omega^2 \right)  \\ 
& \leqslant 4 C t^{-2} \sigma^{-4} \tau^{-3} \rho^{-2} \left( |\nabla \Psi|_\omega^2 + |\overline{\nabla} 
\Psi |_\omega^2 \right).
    \end{aligned}
\end{equation*}
Putting this all together, we derive 
\begin{equation}
    \begin{aligned}
            \Delta_\omega  Q &  \geqslant 
            - C t^2 \sigma^5 \tau^4  \frac{|\Rm(\omega)|_\omega}{A- \widetilde{S}}  - C \sigma^2 \tau \frac{|\Rm(\omega)|_\omega^2}{A- \widetilde{S}} - C \frac{|\Rm(\omega)|_\omega^3}{A - \widetilde{S}} \\
            & + \frac{1}{2}t^{-2}\sigma^{-4} \tau^{-3} \rho^{-2}(|\nabla \Psi|_\omega^2 + |\overline{\nabla} \Psi|_\omega^2) \frac{|\Rm(\omega)|_\omega^2}{(A-\widetilde{S})^2} \\
            & + \frac{1}{2}t^{-2}\sigma^{-4} \tau^{-3} \rho^{-2} (|\nabla \Psi|_\omega^2 + |\overline{\nabla} \Psi|_\omega^2) \frac{|\Rm(\omega)|_\omega^2}{(A-\widetilde{S})^2}\left(1 - \frac{16}{A- \widetilde{S}}\right).
           \end{aligned}
\end{equation}
We can choose $A$ large enough to force a sign on the last term. Then converting the second-last term into quadratic curvature terms via 
\begin{equation}
|\overline{\nabla} \Psi|_{\omega}^{2} \geqslant \frac{1}{2}|\Rm(\omega)|_{\omega}^{2}-C \sigma^4 \tau^2,
\end{equation}
we reach 
\begin{equation}
    \begin{aligned}
            \Delta_\omega Q &\geqslant - Ct^2 \sigma^5 \tau^4 \frac{|\Rm(\omega)|_\omega}{A - \widetilde{S}} \\ 
            & - C ( \sigma^2 \tau + t^{-2}\tau^{-1}\rho^{-2}) \frac{|\Rm(\omega)|_\omega^2}{A - \widetilde{S}} - C \frac{|\Rm(\omega)|_\omega^3}{A - \widetilde{S}} \\ 
            & + \frac{1}{4}t^{-2} \sigma^{-4} \tau^{-3} \rho^{-2}\frac{|\Rm(\omega)|_\omega^4}{(A - \widetilde{S})^2},
    \end{aligned}
\end{equation}
for some uniform $C>0$ and for all $t \geqslant 1$. Observe that at the maximum point of $Q$, we have 
$$ C  ( \sigma^2 \tau + t^{-2}\tau^{-1}\rho^{-2}) \frac{|\Rm(\omega)|_\omega^2}{(A - \widetilde{S})^2} \leqslant \frac{1}{12} t^{-2} \sigma^{-4} \tau^{-3} \rho^{-2}\frac{|\Rm(\omega)|_\omega^4}{(A - \widetilde{S})^2} $$
since otherwise, there would exists $C>0$ such that  $|\Rm(\omega)|_\omega \leqslant C t \sigma^3 \tau^2 \rho$ for all $t \geqslant 1$, leading to a better bound than \eqref{mainineq2}. Similarly, we can assume that at the maximum point of $Q$,
$$ 
    Ct^2\sigma^5 \tau^4 \frac{|\Rm(\omega)|_\omega}{A - \widetilde{S}} \leqslant  \frac{1}{12} t^{-2} \sigma^{-4} \tau^{-3} \rho^{-2} \frac{|\Rm(\omega)|_\omega^4}{(A - \widetilde{S})^2}.
    $$
Otherwise we would obtain $ |\Rm(\omega)|_\omega \leqslant C t^\frac{4}{3}\sigma^{3} \tau^{\frac{7}{3}} \rho^{\frac{2}{3}} $ for all $t \geqslant 1$. Therefore,
\begin{equation}
 \Delta_\omega Q \geqslant - C \frac{|\Rm(\omega)|_\omega^3}{A - \widetilde{S}} + \frac{1}{12} t^{-2} \sigma^{-4} \tau^{-3} \rho^{-2}\frac{|\Rm(\omega)|_\omega^4}{(A - \widetilde{S})^2}.
\end{equation}
At the maximum of $Q$, we have  
$$ 
|\Rm(\omega)|_\omega \leqslant Ct^2 \sigma^4 \tau^3 \rho^2. 
$$
Using the uniform bound for $A- \widetilde{S}$, we obtain \eqref{mainineq2}.
\end{proof}

\section{Applications and Further Remarks}\label{s5}
The usefulness of Theorem \ref{thm2} depends on our ability to construct well-behaved solutions to the continuity equation. Unfortunately, like with the K\"ahler Ricci flow, such solutions are  difficult to construct. In certain cases, however, such as when $X$ is Calabi-Yau or a product manifold, such constructions can be made and thus we can use Theorem \ref{thm2} to understand the continuity method starting from a general initial metric. This was previously done in \cite{Z20} for the K\"ahler-Ricci flow where several results from \cite{TZ15} were reproved. In this section, we study the same setups from these papers but in the context of the continuity method. 

We first recall the following definition. 
\begin{defn}
A long time solution $\omega(t)$ of the continuity method \eqref{cont} on $X$, is of singularity type III  if 
$$
\limsup _{t \rightarrow \infty}\left(\sup _{X}|R m(\omega(t))|_{\omega(t)}\right)<\infty.
$$
Otherwise, we say the solution is of type IIb. 
\end{defn}

\begin{exam}
Suppose that $X$ is a Calabi-Yau Manifold. Thanks to Yau's solution to the Calabi conjecture, we can guarantee the existence of $\omega_{CY}$ on $X$. One can easily check that the solution to the continuity equation starting from $\widetilde{\omega}(0) = \omega_{CY}$ is 
\begin{equation}
    \widetilde{\omega}(t) = \frac{\omega_{CY}}{1+t}. 
\end{equation}
Moreover 
\begin{equation}\label{check1}
    \frac{1}{t} \tr_{\widetilde{\omega}} \widetilde{\omega}_0 =  \frac{1+t}{t}\tr_{\omega_{CY}} \omega_{CY} = \frac{1+t}{t} n \not\rightarrow 0  \text{ as } t \rightarrow \infty.
\end{equation}
Since the curvature re-scales with time as  
\begin{equation}
    |\Rm(\widetilde{\omega}(t))|_{\widetilde{\omega}(t)}^2 = (1+t)^2 |\Rm(\omega_{CY})|_{\omega_{CY}}^2, 
\end{equation}
the solution will be Type III if and only if $\omega_{CY}$ is flat. Furthermore, 
\begin{equation}
    |\nabla_{\widetilde{\omega}(t)}\Rm(\widetilde{\omega}(t))|_{\widetilde{\omega}(t)}^2= (1+t)^2|\nabla_{CY} \Rm(\omega_{CY})|^2_{\omega_{CY}} = 0 \leqslant 1.
\end{equation}
Therefore, for any solution to \eqref{cont} starting from a general metric, we can apply Theorem \ref{thm1} and Theorem \ref{thm2} by $\tau \equiv 1 $ and  $\rho \equiv 1$. Thus, there exists, $C>0$ such that  
\begin{equation}
    |\Rm(\omega(t))|_{\omega(t)} \leqslant Ct^{2},
\end{equation}
for $t \geqslant 1$. 

\end{exam}

\begin{exam}
Let $X$ be a product manifold given by
    $$ X = Y \times Z $$
where $Y$ is a Calabi-Yau manifold and $Z$ is a compact manifold with ample $K_Y$. Define the product metric 
\begin{equation}
    \omega_X = \omega_Y + \omega_Z, 
\end{equation}
where $\omega_Y$ is Calabi-Yau and $\omega_Z$ is a K\"ahler-Einstein metric. We denote the corresponding levi-civita connections by $\nabla_{\omega_Y}$ and $\nabla_{\omega_Z}$ respectively. Starting from $\omega_X$, the family of metrics 
\begin{equation}
    \widetilde{\omega}(t):= \frac{1}{1+t}  \omega_Y + \omega_Z 
\end{equation}
satisfies the continuity equation. Furthermore, the product structure allows us to easily compute the Riemannian Curvature tensor which splits as 
\begin{equation}
    |\Rm(\widetilde{\omega}(t))|_{\widetilde{\omega}(t)}^2 = (1+t)^2 |\Rm(\omega_Y)|_{\omega_Y}^2 + |\Rm (\omega_Z)|_{\omega_Z}^2.
\end{equation}
Therefore, the metric $\omega(t)$ is Type III if and only if $\omega_Y$ is flat. Moreover, 
\begin{equation}
        |\nabla_{\widetilde{\omega}} \Rm(\widetilde{\omega}(t))|_{\widetilde{\omega}(t)}^2 =  |\nabla_{\omega_Z}\Rm (\omega_Z)|_{\omega_Z}^2  \leqslant C.
\end{equation}

Checking our last assumption, we have 
\begin{equation}
    \frac{1}{t} \tr_{\widetilde{\omega}} \omega_0 =  \frac{1+t}{t}\dim(Y) + o(1),
\end{equation}
thus $t^{-1} \tr_{\widetilde{\omega}} \omega_0 \not \rightarrow 0$ as $t \rightarrow \infty$. By Theorem \ref{thm2}, we conclude that any solution of the continuity method on $X$ with $Y$ being a flat tori, there exists $C>0$ such that 
\begin{equation}
    |\Rm(\omega(t))|_{\omega(t)} \leqslant Ct^{2}
\end{equation} 
for all $t \geqslant 1$. 
\end{exam}

\bibliography{main}

\end{document}

%% file: macros.tex
\usepackage{amsmath}
\usepackage{amsfonts}
\usepackage{amssymb}
\usepackage{lipsum}
\usepackage{physics}

\usepackage{amsthm}
\usepackage{enumerate}
\usepackage{amsrefs}

\usepackage{graphicx}
\usepackage{tikz}

\usepackage[titletoc]{appendix}
\usepackage[english]{babel}
\usepackage[utf8]{inputenc}
\usepackage[none]{hyphenat}

\usepackage{geometry}
\usepackage{hyperref}
\hypersetup{
    colorlinks=true,
    linkcolor=blue,
    citecolor=blue,
    filecolor=magenta,      
    urlcolor=cyan,
}

\newcommand\blfootnote[1]{%
  \begingroup
  \renewcommand\thefootnote{}\footnote{#1}%
  \addtocounter{footnote}{-1}%
  \endgroup
}

\newcounter{foo}

\theoremstyle{plain}
\newtheorem{thm}[foo]{Theorem}
\newtheorem{prop}[foo]{Proposition}
\newtheorem{lem}[foo]{Lemma}
\newtheorem{cor}[foo]{Corollary}

\theoremstyle{definition}
\newtheorem{defn}[foo]{Definition}
\newtheorem{exam}[foo]{Example}

\theoremstyle{remark}
\newtheorem{rem}[foo]{Remark}

\swapnumbers

\newcommand{\p}{\partial}
\newcommand{\bp}{\overline{\partial}}

\newcommand{\Ric}{\text{Ric}}
\newcommand{\Rm}{\text{Rm}}

\newcommand{\RR}{\mathbb{R}}
\newcommand{\CC}{\mathbb{C}}
\newcommand{\PP}{\mathbb{P}}

\usepackage{hyperref}
\usepackage{cleveref}


%% file: main.bbl
\begin{bibdiv}
\begin{biblist}

\bib{C86}{techreport}{
      author={Cao, Huai-Dong},
       title={Deformation of kahler metrics to kahler-einstein metrics on
  compact kahler manifolds},
 institution={Princeton Univ.},
        date={1986},
}

\bib{CL21}{article}{
      author={Chu, Jianchun},
      author={Lee, Man-Chun},
       title={On the h{\"o}lder estimate of k{\"a}hler-ricci flow},
        date={2021},
     journal={arXiv preprint arXiv:2105.01602},
}

\bib{DP10}{article}{
      author={Demailly, Jean-Pierre},
      author={Pali, Nefton},
       title={Degenerate complex monge--amp{\`e}re equations over compact
  k{\"a}hler manifolds},
        date={2010},
     journal={International Journal of Mathematics},
      volume={21},
      number={03},
       pages={357\ndash 405},
}

\bib{EGZ08}{article}{
      author={Eyssidieux, Philippe},
      author={Guedj, Vincent},
      author={Zeriahi, Ahmed},
       title={A priori $l^\infty$-estimates for degenerate complex
  monge--amp{\`e}re equations},
        date={2008},
     journal={International Mathematics Research Notices},
      volume={2008},
}

\bib{FZ20}{article}{
      author={Fong, Frederick Tsz-Ho},
      author={Zhang, Yashan},
       title={Local curvature estimates of long-time solutions to the
  k{\"a}hler-ricci flow},
        date={2020},
     journal={Advances in Mathematics},
      volume={375},
       pages={107416},
}

\bib{GTZ13}{article}{
      author={Gross, Mark},
      author={Tosatti, Valentino},
      author={Zhang, Yuguang},
       title={Collapsing of abelian fibered calabi--yau manifolds},
        date={2013},
     journal={Duke Mathematical Journal},
      volume={162},
      number={3},
       pages={517\ndash 551},
}

\bib{GTZ20}{article}{
      author={Gross, Mark},
      author={Tosatti, Valentino},
      author={Zhang, Yuguang},
       title={Geometry of twisted k{\"a}hler--einstein metrics and collapsing},
        date={2020},
     journal={Communications in Mathematical Physics},
      volume={380},
      number={3},
       pages={1401\ndash 1438},
}

\bib{J20}{article}{
      author={Jian, Wangjian},
       title={Convergence of scalar curvature of k{\"a}hler-ricci flow on
  manifolds of positive kodaira dimension},
        date={2020},
     journal={Advances in Mathematics},
      volume={371},
       pages={107253},
}

\bib{JS21}{article}{
      author={Jian, Wangjian},
      author={Song, Jian},
       title={Diameter and ricci curvature estimates for long-time solutions of
  the kahler-ricci flow},
        date={2021},
     journal={arXiv preprint arXiv:2101.04277},
}

\bib{LT16}{article}{
      author={La~Nave, Gabriele},
      author={Tian, Gang},
       title={A continuity method to construct canonical metrics},
        date={2016},
     journal={Mathematische Annalen},
      volume={365},
      number={3},
       pages={911\ndash 921},
}

\bib{R08}{article}{
      author={Rubinstein, Yanir~A},
       title={Some discretizations of geometric evolution equations and the
  ricci iteration on the space of k{\"a}hler metrics},
        date={2008},
     journal={Advances in Mathematics},
      volume={218},
      number={5},
       pages={1526\ndash 1565},
}

\bib{SW20}{article}{
      author={Sherman, Morgan},
      author={Weinkove, Ben},
       title={The continuity equation, hermitian metrics and elliptic bundles},
        date={2020},
     journal={The Journal of Geometric Analysis},
      volume={30},
      number={1},
       pages={762\ndash 776},
}

\bib{ST06}{article}{
      author={Song, Jian},
      author={Tian, Gang},
       title={The k{\"a}hler-ricci flow on surfaces of positive kodaira
  dimension},
        date={2006},
     journal={arXiv preprint math},
        ISSN={0602150/},
}

\bib{ST12}{article}{
      author={Song, Jian},
      author={Tian, Gang},
       title={Canonical measures and k{\"a}hler-ricci flow},
        date={2012},
     journal={Journal of the American Mathematical Society},
      volume={25},
      number={2},
       pages={303\ndash 353},
}

\bib{ST17}{article}{
      author={Song, Jian},
      author={Tian, Gang},
       title={The k{\"a}hler--ricci flow through singularities},
        date={2017},
     journal={Inventiones mathematicae},
      volume={207},
      number={2},
       pages={519\ndash 595},
}

\bib{STZ19}{article}{
      author={Song, Jian},
      author={Tian, Gang},
      author={Zhang, Zhenlei},
       title={Collapsing behavior of ricci-flat kahler metrics and long time
  solutions of the kahler-ricci flow},
        date={2019},
     journal={arXiv preprint arXiv:1904.08345},
}

\bib{SW13}{incollection}{
      author={Song, Jian},
      author={Weinkove, Ben},
       title={An introduction to the k{\"a}hler--ricci flow},
        date={2013},
   booktitle={An introduction to the k{\"a}hler-ricci flow},
   publisher={Springer},
       pages={89\ndash 188},
}

\bib{T19}{article}{
      author={Tian, Gang},
       title={Some progresses on k{\"a}hler--ricci flow},
        date={2019},
     journal={Bollettino dell'Unione Matematica Italiana},
      volume={12},
      number={1},
       pages={251\ndash 263},
}

\bib{TZ21}{article}{
      author={Tian, Gang},
      author={Zhang, Zhenlei},
       title={Relative volume comparison of ricci flow},
        date={2021},
     journal={Science China Mathematics},
       pages={1\ndash 14},
}

\bib{TWX18}{article}{
      author={Tosatti, Valentino},
      author={Weinkove, Ben},
      author={Yang, Xiaokui},
       title={The k{\"a}hler-ricci flow, ricci-flat metrics and collapsing
  limits},
        date={2018},
     journal={American Journal of Mathematics},
      volume={140},
      number={3},
       pages={653\ndash 698},
}

\bib{TZ15}{article}{
      author={Tosatti, Valentino},
      author={Zhang, Yuguang},
       title={Infinite-time singularities of the k{\"a}hler--ricci flow},
        date={2015},
     journal={Geometry \& Topology},
      volume={19},
      number={5},
       pages={2925\ndash 2948},
}

\bib{yZ19}{article}{
      author={Zhang, Yashan},
       title={Collapsing limits of the k{\"a}hler--ricci flow and the
  continuity method},
        date={2019},
     journal={Mathematische Annalen},
      volume={374},
      number={1},
       pages={331\ndash 360},
}

\bib{Z20}{article}{
      author={Zhang, Yashan},
       title={Infinite-time singularity type of the k{\"a}hler--ricci flow},
        date={2020},
     journal={The Journal of Geometric Analysis},
      volume={30},
      number={2},
       pages={2092\ndash 2104},
}

\bib{ZZ19}{article}{
      author={Zhang, Yashan},
      author={Zhang, Zhenlei},
       title={The continuity method on minimal elliptic k{\"a}hler surfaces},
        date={2019},
     journal={International Mathematics Research Notices},
      volume={2019},
      number={10},
       pages={3186\ndash 3213},
}

\bib{ZZ20}{article}{
      author={Zhang, Yashan},
      author={Zhang, Zhenlei},
       title={The continuity method on fano fibrations},
        date={2020},
     journal={International Mathematics Research Notices},
      volume={2020},
      number={22},
       pages={8697\ndash 8728},
}

\end{biblist}
\end{bibdiv}
